\documentclass[12pt,aps,prd]{revtex4}
\usepackage{amsmath,amssymb,amsthm}
\usepackage[makeroom]{cancel}

\newtheorem{theorem}{Theorem}[section]

\newtheorem{proposition}[theorem]{Proposition}
\newtheorem{corollary}[theorem]{Corollary}
\newtheorem{definition}{Definition}

\newenvironment{remark}[1][Remark]{\begin{trivlist}
\item[\hskip \labelsep {\bfseries #1}]}{\end{trivlist}}

\begin{document}

\title{On the axioms of Leibniz algebroids associated to Nambu-Poisson manifolds}
\author{S. Srinivas Rau}
\email{srinivasrau@ifheindia.org}
\affiliation{Department of Mathematics, Faculty of Science and Technology, The ICFAI Foundation for Higher Eduaction, Dontanapally, Hyderabad, 501 203.}
\author{T. Shreecharan}
\email{shreecharan@ifheindia.org}
\affiliation{Department of Physics, Faculty of Science and Technology, The ICFAI Foundation for Higher Eduaction, Dontanapally, Hyderabad, 501 203.}
\begin{abstract}
\begin{center}
{\small ABSTRACT}
\end{center}
Let $E \rightarrow M$ be a smooth vector bundle with a bilinear product on $\Gamma(E)$ satisfying the Jacobi identity. Assuming only the existence of an anchor map $\mathfrak{a}$ we show that $\mathfrak{a}([X,Y]) = [\mathfrak{a}X,\mathfrak{a}Y]_c$. This gives the redundancy of the homomorphism condition in the definition of Leibniz algebroid; in particular if it arises from a Nambu-Poisson manifold.
\end{abstract}
%
%
%
\maketitle

\section{Introduction}

Lie algebroids that were introduced as a generalisation of the Lie algebra idea have been slowly but steadily increasing their appearance in physics. They have been extensively used in the study of classical systems for the past two decades \cite{Weinstein,Liberman}. Recently Lie algebroids have been used to formulate more general gauge theories than Yang-Mills \cite{Strobl2}. This approach has yielded rich dividends. Poisson sigma model \cite{Schaller,Ikeda} a prototype of a Lie algebroid gauge theory has provided a field theoretic insight into the deformation quantisation scheme of Kontsevich. It has also shown some promising glimpses of uniting gravity and gauge theory in a common framework \cite{Strobl1} atleast only in two dimensions as of now.

Recently Loday introduced the concept of Leibniz algebroid that is a natural generalisation of a Lie algebroid by discarding the skew-symmetric condition.
It is worth recalling the definition of a Leibniz algebroid \cite{Ibanez}:
\begin{definition}
A Leibniz algebra structure on a real vector space $\mathfrak{g}$ is a $\mathbb{R}$-bilinear map $[[\, , \,]]:\mathfrak{g} \times \mathfrak{g} \rightarrow \mathfrak{g}$ satisfying the Leibniz identity
\begin{equation*}
[[a_1, [[a_2,a_3]] \ ]] - [[ \ [[a_1,a_2]], a_3]] - [[a_2, [[a_1,a_3]] \ ]] = 0 \quad \mathrm{for} \quad a_1,a_2,a_3 \in \mathfrak{g}
\end{equation*}
\end{definition}
\begin{definition} \label{leibalg}
A Leibniz algebroid structure on a differentiable vector bundle $E \rightarrow M$ is a pair that consists of a Leibniz algebra structure $[[\, , \,]]$ on the space $\Gamma(E)$ of the global cross sections of $E \rightarrow M$ and a vector bundle morphism $\varrho: E \rightarrow TM$, called the anchor map, such that the induced map $\varrho: \Gamma(E) \rightarrow \Gamma(TM) = X(M)$ satisfies the following relations:
\begin{enumerate}
\item \textit{$\varrho[[s_1,s_2]] = [\varrho(s_1),\varrho(s_2)]$}
\item \textit{$[[s_1,fs_2]] = f[[s_1,s_2]] + \varrho(s_1)(f)s_2$}
\end{enumerate}
$\forall \, s_1, s_2 \in \Gamma(E)$ and $f \in \mathcal{C}^\infty(M)$. \\
A triple $(E, [[\, , \,]], \varrho)$ is called a Leibniz algebroid over the manifold $M$.
\end{definition}
Leibniz algebroid has been associated with Nambu-Poisson manifold. In fact this association is very interesting since it has been shown that Nambu-Poisson manifold has atleast two different Leibniz algebroid structures. The first one being derived in Ref \cite{Ibanez} and the other in Ref \cite{Hagiwara}. This is not only interesting from a mathematical point of view but physically also it throws up intriguing questions. We reproduce some of the definition and also these two distinct structures, here for the sake of convenience of the readers.
\begin{definition}
Let $\mathcal{M}$ be a smooth $n$-dimensional manifold. A Nambu-Poisson structure on $\mathcal{M}$ of order $p$ (with $2\leq p\leq n$) is given by a $p$-vector field which satisfies the fundamental identity.
\end{definition}
\begin{definition}
Let $\mathcal{M}$ be a Nambu-Poisson manifold. A Leibniz algebroid attached to $\mathcal{M}$ is the triple $(\bigwedge^{n-1}(T^\ast \mathcal{M}), [[\, , \,]], \Pi)$, where $[[\, , \,]]: \Omega^{n-1}(\mathcal{M}) \times \Omega^{n-1}(\mathcal{M}) \rightarrow \Omega^{n-1}(\mathcal{M})$ is the bracket of $(n-1)$ forms, as defined by Iba\~{n}ez et. al. \cite{Ibanez}
\begin{equation}
[[\alpha,\beta]] = \mathcal{L}_{\Pi(\alpha)} \beta + (-1)^n (\Pi(d\alpha))\beta
\end{equation}
or, as defined by Hagiwara \cite{Hagiwara}
\begin{equation}
[[\alpha,\beta]] = \mathcal{L}_{\Pi(\alpha)} \beta - \imath_{\Pi(\beta)} d\alpha
\end{equation}
for $\alpha, \beta \in \Omega^{n-1}(\mathcal{M})$ and $\Pi:\bigwedge^{n-1}(T^\ast \mathcal{M}) \rightarrow T\mathcal{M}$ if the homomorphism of the vector bundles given by $\Pi(\beta)=i(\beta)\Lambda(x)$; $\Lambda$ being the Nambu-Poisson $n$-vector, $\mathcal{L}$ the Lie derivative, and $\imath$ the interior product. 
\end{definition}
In this work we concentrate on the axioms that form the definition of a Leibniz algebroid. Usually the homomorphism condition of the anchor map is taken to be a part of the definition for a Leibniz. We show that this is a redundant condition. In fact this redundancy has been pointed out for Lie algebroids in some earlier works \cite{Grabowski,Marle}. Our derivation gives yet another proof of this redundancy in the Lie algebroid case.

\section{Leibniz Algebroid}

\noindent Suppose $E \rightarrow M$ is a smooth vector bundle. Let $[\, , \,]$ be a bilinear bracket on the vector space of smooth sections $\Gamma(E)$. Note that $\Gamma(E)$ is a faithful module over the ring $\mathcal{C}^\infty(M)$. We assume
\begin{enumerate}

\item $[X,[Y,Z]] = [[X,Y],Z] + [[Y,[X,Z]] \quad \forall \quad X,Y,Z \in \Gamma(E)$.

\item Let $\mathcal{T}(\mathcal{C}^\infty(M), \mathcal{C}^\infty(M))$ be the set of transformations (self-mappings) of $\mathcal{C}^\infty(M)$. Suppose there is a map $\mathfrak{a}: \Gamma(E) \rightarrow \mathcal{T}$ such that $\forall \, f \in \mathcal{C}^\infty(M)$ and $X, Y \in \Gamma(E)$ one has \\
    $(\mathfrak{a}(X)f)Y = [X,fY] - f[X,Y]$ \\
Since $\Gamma(E)$ is faithful, $(\mathfrak{a}(X)f)Y = 0 \quad \forall \quad Y$ iff $(\mathfrak{a}(X)f) =0 \, \in \, \mathcal{C}^\infty(M)$

\end{enumerate}

\noindent Then
\begin{proposition} \label{prop1}
\begin{enumerate}

\item $\mathfrak{a}([X,Y]) = [\mathfrak{a}(X),\mathfrak{a}(Y)]_c$ where $[ \, , \,]$ is the commutator in $\mathcal{T}: [T, S]_c \, g = T(Sg) - S(Tg)$ $\forall \, g \in \mathcal{C}^\infty(M), \, T,S \, \in \mathcal{T}$.

\item If $\mathfrak{a}$ is a linear map then $\mathfrak{a}([X,Y]) = - \mathfrak{a}([Y,X])$.

\item Each transformation $\mathfrak{a}(X)$ satisfies $\mathfrak{a}(X)(fg) = f(\mathfrak{a}(X)g) + (\mathfrak{a}(X)f)g$ $\forall \, f, g \in \mathcal{C}^\infty(M)$. Thus if $\mathfrak{a}(X)$ is linear then it is a derivation of $\mathcal{C}^\infty(M)$.

\end{enumerate}
\end{proposition}
\begin{proof}
\begin{enumerate}
\item To show that $\mathfrak{a}$ preserves brackets, we choose and fix $f \in \mathcal{C}^\infty(M)$ and $Z \in \Gamma(E)$. Now we claim that $(\mathfrak{a}([X,Y])f)Z = ([\mathfrak{a}(X),\mathfrak{a}(Y)]_c \, f)Z$. By the faithfulness of the module $\Gamma(E)$ we have equality of the brackets for arbitrary $f$. \\
Now the LHS is
\begin{eqnarray*}
(\mathfrak{a}([X,Y])f)Z & = & [[X,Y],fZ] - f [[X,Y],Z] \\
& = & [[X,Y],fZ] - f \big\{[X,[Y,Z]] - [Y,[X,Z]]\big\} \\
& = & [[X,Y],fZ] - f [X,[Y,Z]] + f [Y,[X,Z]]
\end{eqnarray*}
The RHS is
\begin{equation*}
([\mathfrak{a}(X),\mathfrak{a}(Y)]_c \, f)Z = \Big(\mathfrak{a}(X)\big\{\mathfrak{a}(Y)f \big\} - \mathfrak{a}(Y)\big\{\mathfrak{a}(X)f \big\} \Big) Z
\end{equation*}
Setting $\mathfrak{a}(Y)f = g$ and $\mathfrak{a}(X)f = h$ leads to
\begin{eqnarray*}
(\mathfrak{a}([X,Y])f)Z & = & \big(\mathfrak{a}(X)g - \mathfrak{a}(Y)h \big) \\
& = & \big\{[X,gZ] - g[X,Z] \big\} - \{[Y,hZ] - h[Y,Z]\} \\
& = & \big\{[X,\mathfrak{a}(Y)f Z] - \mathfrak{a}(Y)f [X,Z] \big\} - \{[Y,\mathfrak{a}(X)f Z] - \mathfrak{a}(X)f [Y,Z]\} \\
& = & [X,[Y,fZ] -f[Y,Z]] - \big\{[Y,f[X,Z]] - f[Y,[X,Z]] \big\} \\
& - & [Y,[X,f Z]-f[X,Z]] + [X,f[Y,Z]-f[X,[Y,Z]] \\
& = & [X,[Y,fZ]] -\cancel{[X,f[Y,Z]]} - \bcancel{[Y,f[X,Z]]} + f[Y,[X,Z]] \\
& - & [Y,[X,f Z]] + \bcancel{[Y,f[X,Z]]} + \cancel{[X,f[Y,Z]]} - f[X,[Y,Z]] \\
& = & [[X,Y],fZ] + f[Y,[X,Z]] - f[X,[Y,Z]] = \mathrm{LHS} \qed
\end{eqnarray*}
Note that in obtaining the last step we have made use of our assumption (1)

\item If $\mathfrak{a}$ is linear, in particular $\mathfrak{a}(-Z) = - \mathfrak{a}(Z)$ for any $Z$. Therefore
\begin{eqnarray*}
\mathfrak{a}([X,Y]) & = & [\mathfrak{a}(X),\mathfrak{a}(Y)]_c \\
& = & - [\mathfrak{a}(Y),\mathfrak{a}(X)]_c \quad (\mathrm{from \, the \, definition \, of} \, [ \, , \,]_c) \\
& = & - \mathfrak{a}([Y,X]) \qed
\end{eqnarray*}
\item We note that the proof of Skryabin's theorem (Proposition 1.1) \cite{Skryabin}  holds for any $\mathfrak{a}(X)$, which is denoted by $\hat{D}$ by Grabowski (Thm 1, pg 2) \cite{Grabowski}. This gives the Leibniz property: \\
    $\mathfrak{a}(X)(fg) = f(\mathfrak{a}(X)g + (\mathfrak{a}(X)f)g$ \qed \\
By definition a derivation is a linear map on $\mathcal{C}^\infty(M)$ satisfying the Leibniz property \cite{Koszul}.

\end{enumerate}

\end{proof}

\begin{corollary}
If $[X,Y] = -[Y,X] \quad \forall \, X,Y, \in \Gamma(E)$, (i.e., the bracket is antisymmetric) and $\mathfrak{a}$ is a linear map then $\mathfrak{a}$ preserves antisymmetry ie $(\mathfrak{a}[X,Y]) = \mathfrak{a}(-[X,Y])$.
\end{corollary}

\begin{corollary}
In the definition of a Leibniz algebroid (Def 2)  $([[ \, , \,]])$ the bracket-preserving condition (Cond 1) on the anchor map $\varrho$ is redundant. In particular this redundancy holds for Leibniz algebroids arising from Nambu-Poisson manifolds.
\end{corollary}

\begin{remark}
Proposition \ref{prop1} yields another proof of redundancy of bracket-preserving condition in the definition of Lie algebroid.
\end{remark}

\section{Conclusion}

We have shown that the homomorphism condition of the anchor map, that is usually treated as a part of the definition of Leibniz algebroid, is redundant. At the same time a new proof has been provided of this redundancy for the Lie algebroid. As mentioned earlier two different Leibniz algebroid structures have been derived for the Nambu-Poisson manifolds. We hope to return to their mathematical and physical consequences in future.

\end{document}